\documentclass[a4paper,reqno]{amsart}
\usepackage{array,longtable}
\usepackage{amssymb, amsmath, amscd}
\usepackage[final]{graphicx}
\usepackage{float}\usepackage{wrapfig}
\usepackage{color}
\usepackage{bbm}
\usepackage[final]{graphicx}
\usepackage{tikz-cd}
\usepackage[russian,english]{babel}
\usepackage[utf8]{inputenc}
\usepackage[T1,T2A]{fontenc}

\def\Tr{\operatorname{Tr}}
\newtheorem{theorem}{Theorem}[section]
\newtheorem{lemma}[theorem]{Lemma}
\newtheorem{remark}[theorem]{Remark}
\newtheorem{definition}[theorem]{Definition}

\newtheorem{corollary}[theorem]{Corollary}
\makeatletter
 \@addtoreset{equation}{section}
\makeatother
\begin{document}
\title[Volume density asymptotics]
{Volume density asymptotics of\\ central harmonic spaces}
\author{Peter B. Gilkey}
\address{PBG: Mathematics Department, University of Oregon, Eugene OR 97403 USA}
\email{gilkey@uoregon.edu}
\author{JeongHyeong Park}
\address{JHP: Mathematics Department, Sungkyunkwan University, Suwon, 16419 Korea.}
\email{parkj@skku.edu}
\begin{abstract} We show the asymptotics of the volume density function in the class of
central harmonic manifolds can be specified arbitrarily and do not
 determine the geometry.
\end{abstract}
\maketitle
\section{Introduction}
Let $\mathbb{M}_\psi:=(M,\psi^2g)$ be the conformal deformation of a
connected Riemannian manifold $\mathbb{M}:=(M,g)$ of dimension $m\ge4$ where
$\psi$ is a smooth positive function on $M$. If $P$ is a point of $M$,
let $r_{\mathbb{M},P}(Q)$ be the geodesic distance from $P$ to $Q$,
let $\iota_{\mathbb{M},P}$ be the injectivity radius at $P$, let
$B_{\mathbb{M},P}:=\left\{Q\in M:r_{\mathbb{M},P}(Q)<\iota_{\mathbb{M},P}\right\}$ be the open
ball about $P$ of radius $\iota_{\mathbb{M},P}$,
let $\vec e=(e_1,\dots,e_m)$ be orthonormal basis for $T_PM$, and let
$\vec x=(x^1,\dots,x^m):=\exp_P(x^1e_1+\dots+x^me_m)$
be geodesic coordinates centered at $P$. We then have that
$$
r_{\mathbb{M},P}(\vec x)=\|\vec x\|=\left\{(x^1)^2+\dots+(x^m)^2\right\}^{1/2}\,.
$$

Let $\operatorname{dvol}_{\mathbb{M}}$ be the Riemannian measure on $\mathbb{M}$,
let $g_{ij}:=g(\partial_{x^i},\partial_{x^j})$,
and let $d\vec x=dx^1\dots dx^m$ be the Euclidean measure
on $T_PM$. Then
$$
\operatorname{dvol}_{\mathbb{M}}=\tilde\Theta_{\mathbb{M},P} dx^1\dots dx^m\quad\text{where}\quad\tilde\Theta_{\mathbb{M},P}:=\det(g_{ij})^{1/2}
$$
where $\tilde\Theta_{\mathbb{M},P}$ is the {\it volume density function}. Let
$S_P^{m-1}:=\{\vec\theta\in T_PM:\|\vec\theta\|=1\}$ be the unit sphere in $T_PM$ and let
$\mathbb{S}_p^{m-1}:=(S_P^{m-1},g_S)$ where $g_S$ is the metric on $S_P^{m-1}$ induced from
 Euclidean metric on $T_PM$ defined by $g$.
Introduce geodesic polar coordinates $(r,\vec\theta)$ on $B_{\mathbb{M},P}-\{P\}$ to express
$$\vec x=r(\vec x)\vec\theta(\vec x)\quad\text{for}\quad
0<r(\vec x):=\|\vec x\|<\iota_{\mathbb{M},P}\quad\text{and}\quad
\vec\theta(\vec x)=\|\vec x\|^{-1}\vec x\in S_P^{m-1}\,.
$$
Note that $\vec\theta(\vec x)$ is not defined when $x=0$. We may also express
$$
\operatorname{dvol}_{\mathbb{M}}
=\Theta_{\mathbb{M},P} dr\operatorname{dvol}_{\mathbb{S}_P^{m-1}}\quad\text{where}\quad\Theta_{\mathbb{M},P}:=r^{m-1}\tilde\Theta_{\mathbb{M},P}\,.
$$

We say that a smooth function $f$, which is defined near $P$, is {\it radial} if
there exists a smooth function $\eta_1$ of one real variable so $f(\vec x)=\eta_1(\|\vec x\|)$;
$f$ is smooth at $P$ if and only if we can write $f(\vec x)=\eta_2(\|x\|^2)$ or, equivalently,
$\eta_1$ is an even function of $\|\vec x\|$.
We say that $\mathbb{M}$ is {\it central harmonic} at $P$
if $\tilde\Theta_{\mathbb{M},P}$ is a radial function on $B_{\mathbb{M},P}$.
We say that $\mathbb{M}$ is a {\it harmonic space} if $\mathbb{M}$
is central harmonic about every point.

There is a vast literature on this subject; we refer to \cite{B,BTV95,CR40,N05,R31,R63} and the
references cited therein for further details.
Note that if $\mathbb{M}$ is a harmonic space, then we can rescale the metric to replace $g$ by $c^2g$
for any $c>0$ to obtain another harmonic space $\mathbb{M}_c:=(M,c^2g)$.
Similarly, we shall show in Corollary~\ref{C2.2} that if $\mathbb{M}$ is central harmonic at
$P$ and if $\psi$ is a smooth positive radial function, then the
radial conformal deformation $\mathbb{M}_\psi:=(M,\psi^2g)$ is again central harmonic at $P$.

\subsection{The asymptotic expansion of the volume density function}
We can expand the volume density function in geodesic polar coordinates in a formal power-series
$$
\tilde\Theta_{\mathbb{M},P}(r,\vec\theta)\sim
1+\sum_{\nu=2}^\infty\mathcal{H}_\nu(\mathbb{M},P,\vec\theta)r^\nu\,.
$$
If $\xi\in T_PM$, let $\mathcal{J}(\xi):=\mathcal{J}_0(\xi)$ be the Jacobi operator and
$\mathcal{J}_k(\xi):=\nabla_\xi^k\mathcal{J}(\xi)$;
$\mathcal{J}_k(\xi)$ is a self-adjoint endomorphism of $T_PM$ which is characterized by the relationship:
$$
g(\mathcal{J}_k(\xi)\xi_1,\xi_2)=R(\xi_1,\xi,\xi,\xi_2;\xi\dots\xi)\,.
$$
We refer to Gray~\cite{Gr74} for the computation of $\mathcal{H}_\nu(\mathbb{M},P,\vec\theta)$
for $2\le\nu\le 6$ and to  Gilkey and Park~\cite{GP20} for the computation of
$\mathcal{H}_\nu(\mathbb{M},P,\vec\theta)$when $\nu=7,8$.

\begin{theorem}
 Let $P$ be a point of a Riemannian manifold $\mathbb{M}$
 and let $\vec\theta\in S_{\mathbb{M}}^P$.
\begin{enumerate}
\smallbreak\item $\mathcal{H}_{2}(\mathbb{M},P,\vec\theta)=-\frac{\Tr\{\mathcal{J}(\vec\theta)\}}{6}$.
\smallbreak\item $\mathcal{H}_{3}(\mathbb{M},P,\vec\theta)=-\frac{\Tr\{\mathcal{J}_1(\vec\theta)\}}{12}$.
\smallbreak\item $\mathcal{H}_{4}(\mathbb{M},P,\vec\theta)=\frac{\Tr\{\mathcal{J}(\vec\theta)\}^2}{72}
-\frac{\Tr\{\mathcal{J}(\vec\theta)^2\}}{180}-\frac{\Tr\{\mathcal{J}_2(\vec\theta)\}}{40}$.
 \smallbreak\item $\mathcal{H}_{5}(\mathbb{M},P,\vec\theta)=\frac{\Tr\{\mathcal{J}(\vec\theta)\} \Tr\{\mathcal{J}_1(\vec\theta)\}}{72}
 -\frac{\Tr\{\mathcal{J}(\vec\theta)\mathcal{J}_1(\vec\theta)\}}{180}-\frac{\Tr\{\mathcal{J}_3(\vec\theta)\}}{180}$.
 \smallbreak\item $\mathcal{H}_{6}(\mathbb{M},P,\vec\theta)=-\frac{\Tr\{\mathcal{J}(\vec\theta)\}^3}{1296}
 +\frac{\Tr\{\mathcal{J}(\vec\theta)\} \Tr\{\mathcal{J}(\vec\theta)^2\}}{1080}+\frac{\Tr\{\mathcal{J}(\vec\theta)\} \Tr\{\mathcal{J}_2(\vec\theta)\}}{240}$
 \smallbreak\qquad$
 -\frac{\Tr\{\mathcal{J}(\vec\theta)^3\}}{2835}
 -\frac{\Tr\{\mathcal{J}(\vec\theta)\mathcal{J}_2(\vec\theta)\}}{630}+\frac{\Tr\{\mathcal{J}_1(\vec\theta)\}^2}{288}
 -\frac{\Tr\{\mathcal{J}_1(\vec\theta)^2\}}{672}-\frac{\Tr\{\mathcal{J}_4(\vec\theta)\}}{1008}$.
 \goodbreak\smallbreak\item $\mathcal{H}_{7}(\mathbb{M},P,\vec\theta)=\frac{\Tr\{\mathcal{J}(\vec\theta)\} \Tr\{\mathcal{J}(\vec\theta)\mathcal{J}_1(\vec\theta)\}}{1080}
 -\frac{\Tr\{\mathcal{J}(\vec\theta)\}^2 \Tr\{\mathcal{J}_1(\vec\theta)\}}{864}-\frac{\Tr\{\mathcal{J}_5(\vec\theta)\}}{6720}$
 \smallbreak\qquad$+\frac{\Tr\{\mathcal{J}(\vec\theta)\} \Tr\{\mathcal{J}_3(\vec\theta)\}}{1080}
 +\frac{\Tr\{\mathcal{J}(\vec\theta)^2\} \Tr\{\mathcal{J}_1(\vec\theta)\}}{2160}-\frac{\Tr\{\mathcal{J}(\vec\theta)^2\mathcal{J}_1(\vec\theta)\}}{1890}$
 \smallbreak\qquad$
 -\frac{\Tr\{\mathcal{J}(\vec\theta)\mathcal{J}_3(\vec\theta)\}}{3024}+\frac{\Tr\{\mathcal{J}_1(\vec\theta)\} \Tr\{\mathcal{J}_2(\vec\theta)\}}{480}
 -\frac{\Tr\{\mathcal{J}_1(\vec\theta)\mathcal{J}_2(\vec\theta)\}}{1120}$.
 \smallbreak\item $\mathcal{H}_{8}(\mathbb{M},P,\vec\theta)=\frac{\Tr\{\mathcal{J}(\vec\theta)\}^4}{31104}-\frac{\Tr\{\mathcal{J}(\vec\theta)\}^2 \Tr\{\mathcal{J}(\vec\theta)^2\}}{12960}
-\frac{\Tr\{\mathcal{J}(\vec\theta)\}^2 \Tr\{\mathcal{J}_2(\vec\theta)\}}{2880}$
\smallbreak\qquad$+\frac{\Tr\{\mathcal{J}(\vec\theta)\} \Tr\{\mathcal{J}(\vec\theta)^3\}}{17010}
+\frac{\Tr\{\mathcal{J}(\vec\theta)\} \Tr\{\mathcal{J}(\vec\theta)\mathcal{J}_2(\vec\theta)\}}{3780}-\frac{\Tr\{\mathcal{J}_6(\vec\theta)\}}{51840}$
\smallbreak\qquad$
-\frac{\Tr\{\mathcal{J}(\vec\theta)\} \Tr\{\mathcal{J}_1(\vec\theta)\}^2}{1728}
+\frac{\Tr\{\mathcal{J}(\vec\theta)\} \Tr\{\mathcal{J}_1(\vec\theta)^2\}}{4032}$
\smallbreak\qquad$-\frac{\Tr\{\mathcal{J}_2(\vec\theta)^2\}}{7200}
+\frac{\Tr\{\mathcal{J}(\vec\theta)\} \Tr\{\mathcal{J}_4(\vec\theta)\}}{6048}
+\frac{\Tr\{\mathcal{J}(\vec\theta)^2\} \Tr\{\mathcal{J}_2(\vec\theta)\}}{7200}$
\smallbreak\qquad$+\frac{\Tr\{\mathcal{J}(\vec\theta)^2\} ^2}{64800}
-\frac{\Tr\{\mathcal{J}(\vec\theta)^4\}}{37800}-\frac{17 \Tr\{\mathcal{J}(\vec\theta)^2\mathcal{J}_2(\vec\theta)\}}{113400}$
\smallbreak\qquad$
+\frac{\Tr\{\mathcal{J}(\vec\theta)\mathcal{J}_1(\vec\theta)\} \Tr\{\mathcal{J}_1(\vec\theta)\}}{2160}-\frac{5 \Tr\{\mathcal{J}(\vec\theta)\mathcal{J}_1(\vec\theta)^2\}}{18144}
-\frac{\Tr\{\mathcal{J}(\vec\theta)\mathcal{J}_4(\vec\theta)\}}{18144}$
\smallbreak\qquad$+\frac{\Tr\{\mathcal{J}_1(\vec\theta)\} \Tr\{\mathcal{J}_3(\vec\theta)\}}{2160}
-\frac{\Tr\{\mathcal{J}_1(\vec\theta)\mathcal{J}_3(\vec\theta)\}}{5184}+\frac{\Tr\{\mathcal{J}_2(\vec\theta)\}^2}{3200}$.
 \end{enumerate}\end{theorem}

If $\mathbb{M}$ is central harmonic at $P$, then
$\mathcal{H}_\nu(\mathbb{M},P,\vec\theta)$ is independent of
$\vec\theta$; we set $\mathcal{H}_\nu(\mathbb{M},P):=\mathcal{H}_\nu(\mathbb{M},P,\vec\theta)$ for any
$\vec\theta\in S_P^{m-1}$.
Since $\tilde\Theta_{\mathbb{M},P}(\vec x)=\tilde\Theta_{\mathbb{M},P}(-\vec x)$, we may conclude that
$\mathcal{H}_\nu(\mathbb{M},P)=0$ if $\nu$ is odd. If $M$ is a harmonic space, then one can
show that the value is independent of $P$
and we set $\mathcal{H}_\nu(\mathbb{M})=\mathcal{H}_\nu(\mathbb{M},P)$ for any $P$.

\subsection{Specifying the volume density function}
In Section~\ref{S2}, we will use an argument shown
to us by Professor J. \'Alvarez-L\'opez~\cite{AL2022} to
establish the following result.

\goodbreak\begin{theorem}\label{T1.2}\rm
\ \begin{enumerate}
\item Let $\mathbb{M}$ be the germ of a Riemannian manifold which is central harmonic at $P$.
Let $\Xi$ be the germ of a smooth positive function of one real variable.
There exists the germ of a  smooth positive radial function $\psi$ defined on $M$ near $P$ so
that $\tilde\Theta_{\mathbb{M}_\psi,P}=\Xi(r_{\mathbb{M}_\psi})$.
\item If $\Xi\equiv1$, then $\psi$ can be defined on all of $B_{\mathbb{M},P}$.
\item If $\mathbb{M}$ and $\Xi$ are real analytic, then
$\psi$ is real analytic.
\end{enumerate}
\end{theorem}

\subsection{Specifying the volume density asymptotics in even dimensions}
Let $\mathcal{W}_{\mathbb{M}}$ be the {\it Weyl curvature operator},
we refer to Section~\ref{S3.2} for more details.
Let $Q_i$ be points of Riemannian manifolds $\mathbb{M}_i$.
We say that $(\mathbb{M}_1,Q_1)$ is {\it Weyl curvature operator isomorphic} to $(\mathbb{M}_2,Q_2)$ if there
exists a linear isomorphism $\Phi$ from $T_{Q_1}(M_1)$ to $T_{Q_2}(M_2)$ so that
$\mathcal{W}_{\mathbb{M}_1}(Q_1)=\Phi^*\mathcal{W}_{\mathbb{M}_2}(Q_2)$. We say that $\mathbb{M}_1$ is
{\it nowhere Weyl curvature operator isomorphic} to $\mathbb{M}_2$ if $(\mathbb{M}_1,Q_1)$ is not
Weyl curvature operator isomorphic to $(\mathbb{M}_2,Q_2)$ for any points $Q_1\in M_1$ and $Q_2\in M_2$.
{Because $\mathcal{W}_{\mathbb{M}_\psi}=\mathcal{W}_{\mathbb{M}}$ only depends
on the conformal structure, the condition that $\mathbb{M}_1$ is nowhere Weyl curvature
operator isomorphic to $\mathbb{M}_2$ is a conformal condition}.
We introduce the following spaces for certain values of $(i,m)$; they are not defined for all values of $(i,m)$.

\begin{definition}\label{D1.3}\rm Let $\mathbb{M}_{1,m}$
be complex projective space $\mathbb{CP}^k$ if $m=2k$,
let $\mathbb{M}_{2,m}$ be quaternionic projective space $\mathbb{HP}^k$ if $m=4k$,  and  let
$\mathbb{M}_{3,m}$ be the Cayley projective plane $\mathbb{OP}^2$ if $m=16$; we refer to Section~\ref{S1.6} for details.
Let $\mathbb{M}_{4,m}:=\widetilde{\mathbb{CP}}^k$, $\mathbb{M}_{5,m}:=\widetilde{\mathbb{HP}}^k$,
and $\mathbb{M}_{6,m}:=\widetilde{\mathbb{OP}}^2$ be the negative curvature duals in the appropriate dimensions;
we refer to Section~\ref{S1.7} for further details.
These are rank 1 symmetric spaces. Note that
$\mathbb{M}_{2,m}$ and $\mathbb{M}_{5,m}$ are not defined unless $m=4k$,
and $\mathbb{M}_{3,m}$ and $\mathbb{M}_{6,m}$ are not defined if $m\ne16$.
Let $\mathbb{M}_{7,m}=\mathbb{R}^m$
be flat Euclidean space. Let $P_{i,m}\in\mathbb{M}_{i,m}$; the particular point in question is irrelevant
as $\mathbb{M}_{i,m}$ is a homogeneous space. Let $Q_{i,m}$ be arbitrary points of $M_{i,m}$.
Let $\psi_{i,m}$ be a positive function on
$M_{i,m}$ with $\psi_{i,m}(P_{i,m})=1$ which
is radial on $B_{\mathbb{M}_{i,m},P_{i,m}}$ and which satisfies
$\psi_{i,m}(Q_{i,m})=1$ if $r_{i,m}(Q_{i,m})\ge\varepsilon_{i,m}$ for some $0<\varepsilon_{i,m}<\frac12\iota_{i,m}$.
Finally, let $\vec{\mathfrak{H}}:=\{\mathfrak{H}_0,\mathfrak{H}_1,\dots\}$ be a sequence of real numbers with
$\mathfrak{H}_0=1$ and $\mathfrak{H}_\nu=0$ for
$\nu$ odd.
\end{definition}

The asymptotic coefficients $\mathcal{H}_\nu$ have been used to obtain constraints on
the possible geometries of harmonic spaces.
Examining  $\mathcal{H}_2(\mathbb{M},P)$ implies, for example,
that $\mathbb{M}$ is Einstein at $P$. However, the following result, which we
will establish in Section~\ref{S3}, shows that
they do not determine the local geometry of a central harmonic manifold.

\begin{theorem}\label{T1.4}\rm Adopt the notation of Definition~\ref{D1.3}. Let $m\ge4$ be even.
There exist radial functions $\psi_{i,m,\vec{\mathfrak{H}}}$ on $\mathbb{M}_{i,m}$ so $\mathbb{M}_{i,m,\vec{\mathfrak{H}}}:=(M_{i,m},\psi_{i,m,\vec{\mathfrak{H}}}^2g_{i,m})$
satisfies:
\begin{enumerate}
\item If $i\le 3$, $\mathbb{M}_{i,m,\vec{\mathfrak{H}}}$ is
compact.
\item If $4\le i$, $\mathbb{M}_{i,m,\vec{\mathfrak{H}}}$ is diffeomorphic to $\mathbb{R}^m$ and geodesically complete.
\item $\mathbb{M}_{i,m,\vec{\mathfrak{H}}}$ is central harmonic at $P_{i,m}$.
\item $\mathcal{H}_\nu(\mathbb{M}_{i,m,\vec{\mathfrak{H}}},P_{i,m})=\mathfrak{H}_\nu$ for all $\nu$.
\item If $i\ne j$, then $\mathbb{M}_{i,m,\vec{\mathfrak{H}}}$ is nowhere Weyl curvature operator isomorphic to
$\mathbb{M}_{j,m,\vec{\mathfrak{H}}}$;
the local geometries of $\mathbb{M}_{i,m,\vec{\mathfrak{H}}}$ and $\mathbb{M}_{j,m,\vec{\mathfrak{H}}}$
are different everywhere.
\end{enumerate}
\end{theorem}
\subsection{Specifying the volume density asymptotics in odd dimensions}
Any odd dimensional rank~1 symmetric space is conformally flat and thus the rank~1 symmetric
spaces can not be used to extend Theorem~\ref{T1.4} to odd dimensions.
Damek--Ricci spaces are
non-symmetric Hadamard manifolds which are harmonic, but they do not exist in all dimensions
and thus are not adapted to our purposes. Instead, we use Theorem~\ref{T1.2} to
construct odd dimensional examples as follows.

\begin{definition}\label{D1.5}\rm If $m\ge5$ is {\it odd},
let $\mathbb{M}_{i,m-1}=(M_{i,m-1},g_{i,m-1})$
be the {\it even} dimensional Riemannian symmetric space which was
specified in Definition~\ref{D1.3}. Let $B_{i,m-1}$ be a small open geodesic
ball about the basepoint of $M_{i,m-1}$. By Theorem~\ref{T1.2}, we can
choose a radial warping function $\phi_{i,m-1}$ on $B_{i,m-1}$ so that the Riemannian manifold
$\mathbb{N}_{i,m-1}:=(B_{i,m-1},\phi_{i,m-1}^2g_{i,m-1})$
satisfies $\tilde\Theta_{\mathbb{N}_{i,m-1},0}\equiv1$.
Let $dt^2$ be the Euclidean metric on $\mathbb{R}$. Give $B_{i,m-1}\times\mathbb{R}$
the product metric $g_{i,m}:=\phi_{i,m-1}^2g_{i,m-1}\oplus dt^2$ and let $B_{i,m}$
be a small geodesic ball in the resulting Riemannian manifold. Set
$\mathbb{M}_{i,m}:=(B_{i,m},g_{i,m})$.
Since the Weyl conformal curvature operator is a conformal invariant,
$\mathcal{W}_{\mathbb{N}_{i,m-1}}=\mathcal{W}_{\mathbb{M}_{i,m-1}}$.
Since we are considering a product metric, we have
\begin{equation}\label{E1.a}\begin{array}{l}
\tilde\Theta_{\mathbb{M}_{i,m}}
=\tilde\Theta_{\mathbb{N}_{i,m-1}}\cdot{\tilde\Theta}_{\mathbb{R}^1}
\equiv1,\quad\text{and}\\
\mathcal{W}_{\mathbb{M}_{i,m}}
=\mathcal{W}_{\mathbb{N}_{i,m-1}}\oplus\mathcal{W}_{\mathbb{R}^1}
=\mathcal{W}_{\mathbb{M}_{i,m-1}}\oplus0\,.
\end{array}\end{equation}
\end{definition}

\begin{theorem}\label{T1.6}\rm
Adopt the notation of Definitions~\ref{D1.3} and \ref{D1.5}. Let $m\ge5$ be odd.
There exist radial functions $\psi_{i,m,\vec{\mathfrak{H}}}$ on $\mathbb{M}_{i,m}$ so $\mathbb{M}_{i,m,\vec{\mathfrak{H}}}:=(\mathbb{M}_{i,m},\psi_{i,m,\vec{\mathfrak{H}}}^2g_{i,m})$
satisfies:
\begin{enumerate}
\item $\mathbb{M}_{i,m,\vec{\mathfrak{H}}}$ is geodesically complete.
\item $\mathbb{M}_{i,m,\vec{\mathfrak{H}}}$ is central harmonic at $0$ and
$\mathcal{H}_\nu(\mathbb{M}_{i,m,\vec{\mathfrak{H}}},0)=\mathfrak{H}_\nu$ for all $\nu$.
\item If $i\ne j$, then $\mathbb{M}_{i,m,\vec{\mathfrak{H}}}$ is nowhere Weyl curvature operator isomorphic to
$\mathbb{M}_{j,m,\vec{\mathfrak{H}}}$;
the local geometries of $\mathbb{M}_{i,m,\vec{\mathfrak{H}}}$ and $\mathbb{M}_{j,m,\vec{\mathfrak{H}}}$
are different everywhere.
\end{enumerate}
\end{theorem}

\subsection{A 5 dimensional example}
By Lemma~\ref{L2.1}, any radial conformal deformation of a central harmonic
space is again central harmonic. There are, however, central harmonic spaces which do not
arise in this fashion. We will establish the following result in Section~\ref{S2.4}.
\begin{lemma}\label{L1.7}\rm {Adopt the notation of Definition~\ref{D1.5}.}
$\mathbb{M}_{1,5}$ a central harmonic space which
is nowhere Weyl curvature isomorphic to a conformal deformation
of a harmonic space.
\end{lemma}

\subsection{The rank 1 symmetric spaces with positive curvature}\label{S1.6}
Let $\mathbb{S}^m$ be the unit sphere in
$\mathbb{R}^{m+1}$, let $\mathbb{CP}^k$ be
complex projective space, let $\mathbb{HP}^k$ be quaternionic projective space, and let
$\mathbb{OP}^2$ be the Cayley projective plane. We give these spaces the standard metrics
normalized so
\begin{equation}\label{E1.b}
\begin{array}{| l | l | l | l |}\noalign{\hrule}
\mathbb{M}&\text{dimension}&\text{diameter}&\Theta_{\mathbb{M},P}\\\noalign{\hrule}
\mathbb{S}^m&m&\pi&\sin(r)^{m-1}\vphantom{\vrule height 10pt}\\\noalign{\hrule}
\mathbb{CP}^k&2k&\frac12\pi&\sin(r)^{2k-1}\cos(r)\vphantom{A_{\vrule height 9pt}^{\vrule height 7pt}}\\\noalign{\hrule}
\mathbb{HP}^k&4k&\frac12\pi&\sin(r)^{4k-1}\cos^3(r)\vphantom{A_{\vrule height 9pt}^{\vrule height 7pt}}\\\noalign{\hrule}
\mathbb{OP}^2&16&\frac12\pi&\sin(r)^{15}\cos^7(r)\vphantom{A_{\vrule height 9pt}^{\vrule height 7pt}}\\\noalign{\hrule}
\end{array}\end{equation}
The metric on $\mathbb{S}^m$ is the standard metric inherited from Euclidean space,
the metric on $\mathbb{CP}^k$ is the suitably normalized Fubini-Study metric, and
so forth. The rank 1 symmetric spaces in positive curvature
are compact 2 point homogeneous spaces with
$B_{\mathbb{M},P}=M-\mathcal{C}_{\mathbb{M},P}$ where
$\mathcal{C}_{\mathbb{M},P}$ is the cut-locus:
$$\mathcal{C}_{\mathbb{S}^m,P}=\{-P\},\quad
\mathcal{C}_{\mathbb{CP}^k,P}=\mathbb{CP}^{k-1},\quad
\mathcal{C}_{\mathbb{HP}^k,P}=\mathbb{HP}^{k-1},\quad
\mathcal{C}_{\mathbb{OP}^2,P}=S^7\,.
$$
\subsection{The rank 1 symmetric spaces of negative curvature}\label{S1.7}
There are negative curvature duals of the spaces discussed in Section~\ref{S1.6} that we shall
denote by $\widetilde{\mathbb{S}}^m$ (hyperbolic space),
$\widetilde{\mathbb{CP}^k}$ (complex hyperbolic  space),
$\widetilde{\mathbb{HP}}^k$ (quaternionic hyperbolic space), and
$\widetilde{\mathbb{OP}}^2$ (Cayley hyperbolic plane). These are the rank 1 symmetric
spaces of negative curvature; they are all 2-point homogeneous spaces and are
geodesically complete. The curvature tensor of these spaces is obtained by reversing
the sign of the curvature tensor of the corresponding positive curvature example. We note that
any simply-connected 2-point homogeneous space
 is either flat or a rank 1 symmetric space.

If $\mathbb{M}$ is a rank 1 symmetric space with negative curvature, then
the exponential map is a global diffeomorphism so the underlying topology of all these spaces
is Euclidean space; the cut locus is empty.
We adopt the same normalizations as those used to normalize the positive curvature examples.
We replace $\sin$ by $\sinh$ and $\cos$ by $\cosh$ in Equation~(\ref{E1.b}) to obtain
$$
\begin{array}{| l | l | l |}\noalign{\hrule}
\mathbb{M}&\text{d{imension}}&\Theta_{\mathbb{M},P}\\\noalign{\hrule}
\widetilde{\mathbb{S}}^m&m&\sinh(r)^{m-1}
\vphantom{\vrule height 12pt}\\\noalign{\hrule}
\widetilde{\mathbb{CP}}^k&2k&\sinh(r)^{2k-1}\cosh(r)
\vphantom{A_{\vrule height 9pt}^{\vrule height 7pt}}\\\noalign{\hrule}
\widetilde{\mathbb{HP}}^k&4k&\sinh(r)^{4k-1}\cosh^3(r)
\vphantom{A_{\vrule height 9pt}^{\vrule height 7pt}}\\\noalign{\hrule}
\widetilde{\mathbb{OP}}^2&16&\sinh(r)^{15}\cosh^7(r)
\vphantom{A_{\vrule height 9pt}^{\vrule height 7pt}}\\\noalign{\hrule}
\end{array}$$
\subsection{Outline of the paper}
In Section~\ref{S2}, we construct radial conformal deformations of any central harmonic space
realizing any sequence of asymptotic coefficients $\vec{\mathfrak{H}}$ with $\mathfrak{H}_0=1$
and $\mathfrak{H}_\nu=0$ if $\nu$ is odd. We also show that a radial conformal deformation of a central
harmonic space is again central harmonic. In Section~\ref{S3}, we use the Weyl conformal curvature
and the Pontrjagin classes to construct conformal invariants of
the curvature tensor to distinguish the spaces $\mathbb{M}_{i,m,\vec{\mathfrak{H}}}$ of Theorems~\ref{T1.4}
and \ref{T1.6}.

\section{Prescribing the volume density function: The proof of Theorem~\ref{T1.2}}\label{S2}
Let $P$ be a point of a central harmonic Riemannian manifold $\mathbb{M}$.
In Section~\ref{S2.1}, we show that a radial conformal deformation of
$\mathbb{M}$ is again central harmonic and we determine the resulting volume density function. In Section~\ref{S2.2}, we solve the ODE relating the volume density function of a radial
conformal deformation to the original volume density function;
we use this solution in Section~\ref{S2.3} to complete
the proof of Theorem~\ref{T1.2}. In Section~\ref{S2.4}, we establish Lemma~\ref{L1.7}
and determine the warping function $\phi_{1,4}$ on $\mathbb{CP}^2$ to ensure $\tilde\Theta\equiv1$.

\subsection{Radial conformal deformations}\label{S2.1}
Let $\eta(r)$ be a smooth odd function of a single variable with $\dot\eta(0)=1$ and $\dot\eta>0$. Set
$$
\eta_{\mathbb{M}}:=\eta\circ r_{\mathbb{M}},\quad
\psi_{\mathbb{M}}:=\dot\eta\circ r_{\mathbb{M}},\quad
g_\eta:=\psi_{\mathbb{M}}^2g,\quad\text{and}\quad
\mathbb{M}_\eta:=(B_{\mathbb{M},P},g_\eta)\,.
$$
We restrict to $B_{\mathbb{M},P}$ to ensure $r_{\mathbb{M}}^2$ is smooth.
Consequently, since $\dot\eta$ is an even function of $r_{\mathbb{M}}$, $\psi_{\mathbb{M}}$ is a smooth radial
function on $B_{\mathbb{M},P}$ and $\mathbb{M}_\eta$ is a smooth radial conformal deformation of $\mathbb{M}$.
We use an argument introduced previously in Gilkey and Park \cite{GP22} to establish the following
result.

\begin{lemma}\label{L2.1}
\rm Assume that $\mathbb{M}$ is central harmonic at $P$.
\begin{enumerate}
\item $r_{\mathbb{M}_\eta}=\eta_{\mathbb{M}}$.
\item $\tilde\Theta_{\mathbb{M}_\eta}=\eta_{\mathbb{M}}^{1-m}
r_{\mathbb{M}}^{m-1}\psi_{\mathbb{M}}^{m-1}
\tilde\Theta_{\mathbb{M}}$.
\item $\mathbb{M}_\eta$ is central harmonic at $P$.
\end{enumerate}\end{lemma}

\begin{proof} Introduce a system of local coordinates  $\vec\theta=(\theta^1,\dots,\theta^{m-1})$
on $S^{m-1}$ and let $h_{ij}(r,\vec\theta):=g(\partial_{\theta^i},\partial_{\theta^j})$.
We have $d\eta_{\mathbb{M}}=d(\eta\circ r_{\mathbb{M}})
=\{\dot\eta\circ r_{\mathbb{M}}\} dr_{\mathbb{M}}=\psi_{\mathbb{M}}dr_{\mathbb{M}}$ so:
\begin{equation}\label{E2.a}\begin{array}{rl}
g&=dr_{\mathbb{M}}\otimes dr_{\mathbb{M}}
+h_{ij}(r,\vec\theta)d\theta^i\otimes d\theta^j\,,\\[0.05in]
g_\eta&=\psi_{\mathbb{M}}dr_{\mathbb{M}}\otimes\psi_{\mathbb{M}}dr_{\mathbb{M}}
+(\psi_{\mathbb{M}})^2h_{ij}(r,\vec\theta)d\theta^i\otimes d\theta^j\\[0.05in]
&=d\eta_{\mathbb{M}}\otimes d\eta_{\mathbb{M}}
+(\psi_{\mathbb{M}})^2h_{ij}(r,\vec\theta)d\theta^i\otimes d\theta^j\,.
\end{array}\end{equation}
Since $\dot\eta>0$,
the map $Q\rightarrow(\eta_{\mathbb{M}}(Q),\vec\theta(Q))$ introduces new coordinates which,
by Equation~(\ref{E2.a}), are geodesic polar coordinates
centered at $P$ for the metric $g_\eta$. We have reparametrized the
radial parameter to ensure it has unit length and left the angular parameter
unchanged. Assertion~(1) now follows.

Let $\varepsilon(\vec\theta)$ be defined by the identity
$\varepsilon(\vec\theta)\operatorname{dvol}_{\mathbb{S}^{m-1}}(\vec\theta)=
d\theta^1\cdot\cdot\cdot d\theta^{m-1}$;
$\varepsilon$ is independent of the radial parameter. We may express
\begin{equation}\label{E2.b}\begin{array}{l}
\operatorname{dvol}_{\mathbb{M}}=\det(h_{ij})^{1/2}dr_{\mathbb{M}}d\theta^1\cdot\cdot\cdot d\theta^{m-1}
=\det(h_{ij})^{1/2}\varepsilon(\vec\theta)dr_{\mathbb{M}}\operatorname{dvol}_{\mathbb{S}^{m-1}},\\[0.05in]
\Theta_{\mathbb{M}}=\det(h_{ij})^{1/2}\varepsilon(\vec\theta),\quad\text{and}\quad
\tilde\Theta_{\mathbb{M}}=r_{\mathbb{M}}^{1-m}
\det(h_{ij})^{1/2}\varepsilon(\vec\theta)\,.
\end{array}\end{equation}

The angular variable $\vec\theta$ is the same for both systems of geodesic polar coordinates.
We use Equation~(\ref{E2.a}) and Equation~(\ref{E2.b}) to complete the proof by showing:
\medbreak\hfill$\displaystyle
\tilde\Theta_{\mathbb{M}_\eta}=\eta_{\mathbb{M}}^{1-m}\psi_{\mathbb{M}}^{m-1}\det(h_{ij}(r_\mathbb{M},\vec\theta))^{1/2}\varepsilon(\vec\theta)=\eta_{\mathbb{M}}^{1-m}
r_{\mathbb{M}}^{m-1}
\psi_{\mathbb{M}}^{m-1}
\tilde\Theta_{\mathbb{M}}$.\hfill\vphantom{.}
\end{proof}

We have chosen to start with $\eta$, which is the new radial distance function. However, if we
start with a deformation $\Psi$ which is radial on $B_{\mathbb{M},P}$,
then we have $\Psi(\vec x)=\psi(\|\vec x\|)$
for some smooth even function $\psi$ of 1-variable if $r_{\mathbb{M}}<\iota_{\mathbb{M}}$. We set
$$
\eta_\psi(r):=\int_{t=0}^r\psi(t)dt\,.
$$
We then have $\Psi=\psi_{\mathbb{M}}$ on $B_{\mathbb{M},P}$ so the two formalisms are equivalent.
The following observation, which proves Assertion~(3) of Theorem~\ref{T1.4}, is now immediate.

\begin{corollary}\label{C2.2}\rm Let $\mathbb{M}$ be a Riemannian manifold which is central harmonic at $P$.
Let $\Psi$ be a smooth positive function on $\mathbb{M}$ which is radial on $B_{\mathbb{M},P}$. Then the
conformal deformation $(M,\Psi^2g)$ is central harmonic at $P$.
\end{corollary}

\subsection{Solving an ODE}\label{S2.2}
The proof of the following result was shown to us by J. \'Alvarez-L\'opez~\cite{AL2022}.

\begin{lemma}\label{L2.3}\rm Let $f_i(r)$ be positive smooth even functions of one variable which are defined
for $0\le r\le\varepsilon$ and which satisfy $f_i(0)=1$. Then there exists $0<\delta\le\varepsilon$
and a smooth odd function $\eta$
which is defined for $0\le r\le\delta$ so that $\dot\eta(0)=1$ and so that
\begin{equation}\label{E2.c}
f_1(\eta(r))=\eta(r)^{1-m}r^{m-1}\dot\eta(r)^{m-1}f_2(r)\quad\text{for}\quad0\le r\le\delta\,.
\end{equation}
\end{lemma}

\begin{proof} Set
$\phi_i:=f_i^{\frac1{1-m}}$. Then Equation~(\ref{E2.c}) is equivalent to
\begin{equation}\label{E2.d}
\frac1{\phi_1(\eta(r))}=\frac{\dot\eta(r)r}{\eta(r)\phi_2(r)}\quad\text{i.e.}\quad
\frac{\phi_2(r)}r=\frac{\dot\eta(r)\phi_1(\eta(r))}{\eta(r)}\,.
\end{equation}
Multiplying Equation~(\ref{E2.d}) by $dr$ and noting $\dot\eta dr=d\eta$ yields the equivalent relation
\begin{equation}\label{E2.e}
\frac{\phi_2(r)dr}r=\frac{\phi_1(\eta)d\eta}{\eta}\quad\text{i.e.}\quad
\int\frac{\phi_2(r)dr}r=\int\frac{\phi_1(\eta)d\eta}{\eta}+C\,.
\end{equation}
Because $\phi_i$ are even functions with $\phi_i(0)=1$, we may express
$\phi_i(r)=1+r^2\Phi_i(r)$ to rewrite Equation~(\ref{E2.e}) in the form
\begin{equation}\label{E2.f}\begin{array}{l}
\displaystyle\int\frac{(1+r^2\Phi_2(r))dr}r=\int\frac{(1+\eta^2\Phi_1(\eta))d\eta}{\eta_{\vphantom{\vrule height 12pt}}}+C\quad\text{i.e.}\\
\displaystyle\ln|r|+\int r\Phi_2(r)dr=\ln(|\eta|)+\int\eta\Phi_1(\eta)d\eta+C\,.
\end{array}\end{equation}
We set
$$\alpha_i(r):=\int_{t=0}^r t\Phi_i(t)dt\quad\text{and}\quad\eta(r)=r\beta(r)\,.
$$
We then have that $\alpha_i$ is a smooth even function with $\alpha_i(0)=0$.
We may rewrite
Equation~(\ref{E2.f}) in the form
\begin{equation}\label{E2.g}
\ln(|r|)+\alpha_2(r)=\ln(|r|)+\ln(|\beta(r)|)+\alpha_1(r\beta(r))\,.
\end{equation}
Equation~(\ref{E2.g}) is then equivalent to
the relation $G(r,\beta(r))=0$ where
\begin{equation}\label{E2.h}
G(r,\beta):=\alpha_2(r)-\ln(|\beta|)-\alpha_1(r\beta)\,.
\end{equation}
Equation~(\ref{E2.h}) is solved when $r=0$ and $\beta=1$. We compute
 $$
\partial_\beta G(r,\beta)\bigg|_{r=0,\beta=1}=\left.\left\{-\frac1\beta-r\dot\alpha_1(r\beta)\right\}\right|_{r=0,\beta=1}\ne0\,.
 $$
Thus we may use the implicit function theorem
to solve Equation~(\ref{E2.h}) near the point $(r=0,\beta=1)$; the solution is unique and a smooth function of $r$;
if the data is real analytic, then $\beta$ is real analytic. Since
the functions $\alpha_i$ are even functions of $r$, it follows $\beta$ is an even function of $r$ and hence
$\eta$ is an odd function of $r$ with $\dot\eta(0)=1$.
\end{proof}

\subsection{The proof of Theorem~\ref{T1.2}}\label{S2.3}
Assertions~(1) and (3) of Theorem~\ref{T1.2} follow immediately from Lemma~\ref{L2.3}.
If $\Xi=1$, it is not necessary to localize. If we set $f_1\equiv1$ in Lemma~\ref{L2.3},
then $\Phi_1\equiv0$, $\alpha_1\equiv0$, and
Equation~(\ref{E2.g}) simplifies to become $\beta(r)=e^{\alpha_2(r)}$; since $\alpha_2(0)=0$, $\beta(0)=1$.
Thus we can find $\psi$ which is defined on all of $B_{\mathbb{M},P}$
so that $\tilde\Theta_{\mathbb{M}_\psi,0}\equiv1$; it is not necessary to invoke the implicit function
theorem and work locally.\qed

\subsection{The proof of Lemma~\ref{L1.7}}\label{S2.4} {We have by
Definition~\ref{D1.5} that $\mathbb{M}_{1,5}$ is a small geodesic
ball in $\mathbb{N}_{1,4}\times\mathbb{R}$.}
 Equation~(\ref{E1.a}) shows $\mathbb{M}_{1,5}$ is central
 harmonic about the origin and that
 $\mathcal{W}_{\mathbb{M}_{1,5}}$ is nowhere
 vanishing. Nikolayevsky~\cite{N05} has shown that
 every harmonic space of dimension 5 is a space form and
 hence conformally flat. Thus $\mathbb{M}_{1,5}$ is nowhere Weyl curvature isomorphic to
 a radial conformal deformation of a harmonic space; consequently not all central harmonic
 spaces arise as radial conformal deformations of harmonic spaces.\qed

 \medbreak $\mathbb{N}_{1,4}$ is defined by a conformal radial deformation
 $\phi_{1,4}$ of the Fubini-Study metric $B_{\mathbb{CP}^2,P}$ which
is described as follows.

\begin{lemma}\label{L2.4}
\rm Let $\mathbb{M}=\mathbb{CP}^2$ and let
\begin{eqnarray*}
\phi_{1,4}(r):=&&3^{\frac34}e^{\frac{\pi}{2\sqrt3}}\sin (r) \exp\left(-\frac{1}{2} \sqrt{3} \tan ^{-1}\left(\frac{2 \cos ^{\frac{2}{3}}(r)+1}{\sqrt{3}}\right)\right)\\
\\&&\cdot\left(\sqrt{1-\cos ^{\frac{2}{3}}(r)} \sqrt[3]{\cos (r)} \left(\cos ^{\frac{2}{3}}(r)+\cos ^{\frac{4}{3}}(r)+1\right)^{5/4}\right)^{-1}
\end{eqnarray*}
for $r<\frac\pi2$. Then $\phi_{1,4}(0)=1$ and $\tilde\Theta_{\mathbb{M}_{\phi_{1,4}}}=1$ on $B_{\mathbb{CP}^2,P}$.
\end{lemma}

\begin{proof} By Equation~(\ref{E1.b}),
$r^3\tilde\Theta_{\mathbb{CP}^2}=\Theta_{\mathbb{CP}^2}=\sin^3(r)\cos(r)$. Consequently,
Equation~(\ref{E2.c}) becomes $1=\dot\eta^3\eta^{-3}\sin^3\cos(r)$.
Mathematica solves this equation to yield
\begin{eqnarray*}
\eta(r)&=&c_1 \exp \left(\frac12 \log \left(1-\cos ^{\frac{2}{3}}(r)\right)\right)\\
&&\cdot\exp\left(-\frac14\log \left(\cos ^{\frac{2}{3}}(r)+\cos ^{\frac{4}{3}}(r)+1\right)\right)\\
&&\cdot\exp\left(-\frac{\sqrt{3}}2 \tan ^{-1}\left(\frac{2 \cos ^{\frac{2}{3}}(r)+1}{\sqrt{3}}\right)\right)
\end{eqnarray*}
and consequently
\begin{eqnarray*}
\phi_{1,4}(r)&=&c_1\sin (r) \exp\left(-\frac{ \sqrt{3}}{2} \tan ^{-1}\left(\frac{2 \cos ^{\frac{2}{3}}(r)+1}{\sqrt{3}}\right)\right)\\
&&\cdot\left(\sqrt{1-\cos ^{\frac{2}{3}}(r)} \sqrt[3]{\cos (r)} \left(\cos ^{\frac{2}{3}}(r)+\cos ^{\frac{4}{3}}(r)+1\right)^{5/4}\right)^{-1}\,.
\end{eqnarray*}
This is defined for $0<r<\frac\pi2$; there is an apparent singularity at $r=0$ which we
ignore for the moment. We set $c_1=3^{3/4} e^{\frac{\pi }{2 \sqrt{3}}}$ and expand $\phi_{1,4}(r)$ for $r>0$:
\begin{eqnarray*}
\phi_{1,4}(r)&=&1+\frac12r^2+\frac{13}{72}r^4+\frac{1177}{19440}r^6+\frac{7369}{362880}r^8
+\frac{681907}{97977600}r^{10}+O(r^{12})\,.
\end{eqnarray*}
We conclude that $\phi_{1,4}$ is regular at $0$ with $\phi_{1,4}(0)=1$ and thus this is the
radial conformal deformation given by Theorem~\ref{T1.2}.
\end{proof}

\begin{remark}\rm The injectivity radius of $\mathbb{CP}^2$ is $\frac\pi2$. Since $\lim_{r\rightarrow\frac\pi2}\psi(r)=\infty$,
$\psi$ does not extend to all of $\mathbb{CP}^2$. Since
$$\psi(r)\sim\frac{3^{3/4} e^{\frac{\pi }{4 \sqrt{3}}}}{\{\frac\pi2-r\}^{\frac13}}+O(1)\quad\text{as}\quad r\rightarrow\frac\pi2\,,
$$ $\psi$ is integrable
on $[0,\frac\pi2]$ so
the deformed metric on $B_{\mathbb{CP}^2,P}$ is geodesically incomplete.
\end{remark}
\section{Prescribing the volume density asymptotics}\label{S3}

In Section~\ref{S3.1} we establish the first four Assertions of Theorem~\ref{T1.4}
and in Section~\ref{S3a}, we estsablish the first two Assertions of Theorem~\ref{T1.6}.
The heart of the matter, of course, is to distinguish the manifolds $\mathbb{M}_{i,m,\vec{\mathfrak{H}}}$.
In Section~\ref{S3.2}, we review some facts concerning the Weyl conformal curvature
operator. In Section~\ref{SX3}, we complete the proof of
Theorem~\ref{T1.4} and
in Section~\ref{S3.3}, we complete the proof of Theorem~\ref{T1.6}.

\subsection{The proof of Assertions~(1)--(4) of Theorem~\ref{T1.4}}\label{S3.1}
Since the underlying manifold is unchanged, and since $\mathbb{CP}^k$, $\mathbb{HP}^k$, and
$\mathbb{OP}^2$ are compact, it is immediate that $\mathbb{M}_{i,m,\vec{\mathfrak{H}}}$ is compact
for $1\le i\le3$.  If $i\ge4$, then $\mathbb{M}_{i,m}$ is a homogeneous space and hence
geodesically complete. Since the warping function $\psi$ is $1$ outside a compact set,
it is follows that $\mathbb{M}_{i,m,\vec{\mathfrak{h}}}$ is geodesically complete for $4\le i\le 7$. This establishes
Assertions~(1) and (2) of Theorem~\ref{T1.4}.

Assertion~(3) of Theorem~\ref{T1.4} follows from Corollary~\ref{C2.2};
a radial conformal deformation of a central harmonic space is again central harmonic.
Let $\mathfrak{H}=\{\mathfrak{H}_0,\dots\}$
where $\mathfrak{H}_0=1$ and $\mathfrak{H}_\nu=0$ if $\nu$ is odd.
Any formal Taylor series can be realized. Thus, we can find the germ of a smooth positive
even function $\Xi$ of 1-real variable so that $\Xi(t)\sim\sum_{\nu=0}^\infty\mathfrak{H}_\nu t^\nu$.
We apply Theorem~\ref{T1.2} to find the germ of a radial function $\psi$ so $\tilde\Theta_{\mathbb{M}_\psi}(r_{\mathbb{M}_\psi})=\Xi(r_{\mathbb{M}_\psi})$ and thus $\tilde\Theta_{\mathbb{M}_\psi}$ has the right asymptotic
coefficients. By using a partition
of unity, we may assume $\psi(r)=1$ for $r\ge\varepsilon$. Assertion~(4) of Theorem~\ref{T1.4} then follows.

\subsection{The proof of Assertions~(1,2) of Theorem~\ref{T1.6}}\label{S3a} Let $m\ge5$ be odd
and let $\mathbb{M}=(M,g):=\mathbb{M}_{i,m}$.
It is obvious from the definition that $\mathbb{M}$ is central harmonic at the center $0$
of the small geodesic ball defining $M$. We argue as above to find the germ of a radial function
$\psi$ so $\tilde\Theta_{\mathbb{M}_\psi}$ has the right asymptotic coefficients. By using
a partition of unity, we can suppose that $\psi$ grows sufficiently rapidly at the boundary of
$M$ and hence $\mathbb{M}_\psi$ is godesically complete.

\subsection{The Weyl tensor}\label{S3.2}
Let $\rho$ be {the} Ricci tensor and let $\tau$ be {the} scalar curvature.  Let
\begin{eqnarray*}
&&W(x,y,z,w):=R(x,y,z,w)
+\tau\frac{g(x,w)g(y,z)-g(x,z)g
(y,w)}{(m-1)(m-2)}\\
&&\qquad+\frac{\rho(x,z)g(y,w)+\rho(y,w)g(x,w)
-\rho(y,z)g(x,w)-\rho(x,w)g(y,z)}{m-2}
\end{eqnarray*}
be the {\it Weyl conformal curvature tensor}.
The {\it Weyl conformal curvature operator} $\mathcal{W}$ is the skew-symmetric operator
which is characterized by the relation
$$
g(\mathcal{W}(x,y)z,w)=W(x,y,z,w)\,.
$$
The
{\it Weyl Jacobi operator} $\mathcal{J}^W$ is the self-adjoint operator defined by
$$
\mathcal{J}^W(x)y:=\mathcal{W}(y,x)x\,.
$$
We say that $\mathbb{M}$ is {\it conformally flat} if $\mathbb{M}$ is isometric to $\mathbb{R}^m_\psi$ for
some $\psi$. The following result is well known.

\begin{lemma}\label{L3.1}\rm Let $\mathbb{M}$ be a Riemannian manifold of dimension $m\ge4$.
Then
$$
W_{\mathbb{M}_\psi}=\psi^2W_{\mathbb{M}},\quad
\mathcal{W}_{\mathbb{M}_\psi}=\mathcal{W}_{\mathbb{M}},\quad\text{and}\quad
\mathcal{J}^W_{\mathbb{M}_\psi}=\mathcal{J}^W_{\mathbb{M}}\,.
$$
 Furthermore, $\mathbb{M}$ is conformally flat
if and only if $\mathcal{W}_{\mathbb{M}}$ vanishes identically.
\end{lemma}

\subsection{The proof of Theorem~\ref{T1.4}~(4)}\label{SX3}
 We examine the eigenvalue structrure of the Jacobi operator $\mathcal{J}$
and the conformal Jacobi operator $\mathcal{J}^W$of the spaces $\mathbb{M}_{i,m}$ for $m$ even.
\begin{lemma}\label{L3.2}\rm Let $\mathbb{M}=(M,g)$ be a rank 1 symmetric space with positive curvature and
let $x\in T_P(M)$ and $y\in T_P(M)$ be a unit tangent vectors. Let $k\ge2$.
\begin{enumerate}
\item $\mathcal{J}(x)$ is a self-adjoint operator with eigenvalues $\{0,1,4\}$ and corresponding
eigenspace decomposition of $T_P(M)=E_0(x)\oplus E_1(x)\oplus E_4(x)$. Let $y$ be a unit
tangent vector. We have
$$\begin{array}{| l | c | c | c | c |}\noalign{\hrule}
\mathbb{M}&\dim\{E_0(x)\}&\dim\{E_1(x)\}&\dim\{E_4(x)\}&\rho(y,y)\\ \noalign{\hrule}
\mathbb{R}^m&m&0&0&0\\ \noalign{\hrule}
\mathbb{CP}^k&1&2k-2&1&2k+2\\ \noalign{\hrule}
\mathbb{HP}^k&1&4k-4&3&4k+8\\ \noalign{\hrule}
\mathbb{OP}^2&1&8&7&36\\ \noalign{\hrule}
\end{array}.$$
\item The decomposition of Assertion~(1) gives the eigenspace decomposition of
$\mathcal{J}^W(x)$. The corresponding eigenvalues $\lambda_i$ are given by:
$$\begin{array}{| l | c | c | c |}\noalign{\hrule}
\mathbb{M}&\lambda_0&\lambda_1&\lambda_4\\ \noalign{\hrule}
\mathbb{R}^m&0&0&0\\ \noalign{\hrule}
\mathbb{CP}^k&0&1-\frac{2k+2}{2k-1}&4-\frac{2k+2}{2k-1}  \\[0.02in]\noalign{\hrule}
\mathbb{HP}^k&0&1-\frac{4k+8}{4k-1}&4-\frac{4k+8}{4k-1}  \\[0.02in]\noalign{\hrule}
\mathbb{OP}^2&0&1-\frac{36}{15}&4-\frac{36}{15}   \\[0.02in]\noalign{\hrule}
\end{array}$$
\end{enumerate}
\end{lemma}

\begin{proof} The eigenvalue and eigenspace structure of the Jacobi operator for $\mathbb{S}^m$, $\mathbb{CP}^k$,
and $\mathbb{HP}^k$ is well known. The curvature tensor of $\mathbb{OP}^2$ was computed
by Brown and Gray~\cite{BG72} (see  Theorem 6.1); the corresponding eigenvalue decomposition of the Jacobi operator
computed by  Nikolayevsky~\cite{N03} (see the discussion on Page 510). The rank 1 symmetric spaces are Einstein.
For such spaces, the Weyl conformal Jacobi operator is defined by subtracting a
suitable multiple $\kappa$
of the Jacobi
operator for the sphere where {we set $\kappa:=\displaystyle\frac{\rho_{\mathbb{M}}(x,x)}{m-1}$
 to ensure}
$\operatorname{Tr}\{\mathcal{J}^W_{\mathbb{M}}\}(x)=0$ for all unit vectors $x$. Assertion~(2) now follows.
\end{proof}

We reverse the sign of the curvature tensor to compute for the negative curvature duals. The following result is now
immediate from Lemma~\ref{L3.1} and from Lemma~\ref{L3.2}.
\begin{lemma}\label{L3.3}\rm Let $m\ge4$ be even, let
$\mathbb{M}=\mathbb{M}_{i,m,\vec{\mathfrak{H}}}$, let $Q$ be a point of $M$,
and let $0\ne x\in T_Q(M_{i,m})$.
\begin{enumerate}
\item If $i<7$, then $0$ is an eigenvalue of multiplicity $1$ of $\mathcal{J}^W_{\mathbb{M}}(x)$.
If $i=7$, then $\mathcal{J}^W_{\mathbb{M}}(x)$ vanishes identically.
\item If $i=1$, so $M_{i,m}=\mathbb{CP}^k$ for $m=2k$ and $k\ge2$, then
$\mathcal{J}^W_{\mathbb{M}}(x)$ has a negative eigenvalue of multiplicity $m-2$ and
a positive eigenvalue of multiplicity $1$.
\item If $i=2$, so $M_{i,m}=\mathbb{HP}^k$ for $m=4k$ and $k\ge2$, then
$\mathcal{J}^W_{\mathbb{M}}(x)$ has a negative eigenvalue of multiplicity $m-4$ and
a positive eigenvalue of multiplicity $3$.
\item If $i=3$, so $M_{i,m}=\mathbb{OP}^2$ for $m=16$, then
$\mathcal{J}^W_{\mathbb{M}}(x)$ has a negative eigenvalue of multiplicity $8$ and
a positive eigenvalue of multiplicity $7$.
\item If $i=4$, so $M_{i,m}=\widetilde{\mathbb{CP}}^k$ for $m=2k$ and $k\ge2$, then
$\mathcal{J}^W_{\mathbb{M}}(x)$ has a positive eigenvalue of multiplicity $m-2$ and
a negative eigenvalue of multiplicity $1$.
\item If $i=5$, so $M_{i,m}=\widetilde{\mathbb{HP}}^k$ for $m=4k$ and $k\ge2$, then
$\mathcal{J}^W_{\mathbb{M}}(x)$ has a positive eigenvalue of multiplicity $m-4$ and
a negative eigenvalue of multiplicity $3$.
\item If $i=6$, so $M_{i,m}=\widetilde{\mathbb{OP}}^2$ for $m=16$, then
$\mathcal{J}^W_{\mathbb{M}}(x)$ has a positive eigenvalue of multiplicity $8$ and
a negative eigenvalue of multiplicity $7$.\end{enumerate}
\end{lemma}

Assertion~(4) of Theorem~\ref{T1.4} now follows from Lemma~\ref{L3.3}; this completes the proof of Theorem~\ref{T1.4}\qed

\subsection{The proof of Theorem~\ref{T1.6}~(2)}\label{S3.3} Let $m\ge5$ be odd.
Let $y=(x,t)$ be a tangent vector of $M_{i,m}$ where $y$ is a tangent vector to $M_{i,m-1}$ and
$t$ is a tangent vector to $\mathbb{R}$. Since $\mathcal{J}$ is conformal, we may use Equation~(\ref{E1.a}) to see:
\begin{equation}\label{E3.a}
\mathcal{J}^W_{\mathbb{M}_{i,m,\vec{\mathfrak{H}}}}(x)=
\mathcal{J}^W_{\mathbb{M}_{i,m}}(x)=
\mathcal{J}^W_{\mathbb{M}_{i,m-1}}(y)\,.
\end{equation}
The following result now follows from Equation~(\ref{E3.a}) and from Lemma~\ref{L3.3}.

\begin{lemma}\rm Let $m\ge5$ be odd, let
$\mathbb{M}=\mathbb{M}_{i,m,\vec{\mathfrak{H}}}$,  let $Q$ be a point of $M_{i,m}$.
Choose $x\in T_P(M_{i,m})$ so $\mathcal{J}^W_{\mathbb{M}}(x)$ has maximal rank.
\begin{enumerate}
\item If $i<7$, then $0$ is an eigenvalue of multiplicity $2$ of $\mathcal{J}^W_{\mathbb{M}}(x)$.
If $i=7$, then $\mathcal{J}^W_{\mathbb{M}}(x)$ vanishes identically.
\item If $i=1$ so $M_{i,m-1}=\mathbb{CP}^k$ for $m-1=2k$ and $k\ge2$, then
$\mathcal{J}^W_{\mathbb{M}}(x)$ has a negative eigenvalue of multiplicity $m-3$ and
a positive eigenvalue of multiplicity $1$.
\item If $i=2$ so $M_{i,m-1}=\mathbb{HP}^k$ for $m-1=4k$ and $k\ge2$, then
$\mathcal{J}^W_{\mathbb{M}}(x)$ has a negative eigenvalue of multiplicity $m-5$ and
a positive eigenvalue of multiplicity $3$.
\item If $i=3$ so $M_{i,m-1}=\mathbb{OP}^2$ for $m-1=16$, then
$\mathcal{J}^W_{\mathbb{M}}(x)$ has a negative eigenvalue of multiplicity $8$ and
a positive eigenvalue of multiplicity $7$.
\item If $i=4$ so $M_{i,m-1}=\widetilde{\mathbb{CP}}^k$ for $m-1=2k$ and $k\ge2$, then
$\mathcal{J}^W_{\mathbb{M}}(x)$ has a positive eigenvalue of multiplicity $m-3$ and
a negative eigenvalue of multiplicity $1$.
\item If $i=5$ so $M_{i,m-1}=\widetilde{\mathbb{HP}}^k$ for $m-1=4k$ and $k\ge2$, then
$\mathcal{J}^W_{\mathbb{M}}(x)$ has a positive eigenvalue of multiplicity $m-5$ and
a negative eigenvalue of multiplicity $3$.
\item If $i=6$ so $M_{i,m-1}=\widetilde{\mathbb{OP}}^2$ for $m-1=16$, then
$\mathcal{J}^W_{\mathbb{M}}(x)$ has a positive eigenvalue of multiplicity $8$ and
a negative eigenvalue of multiplicity $7$.\end{enumerate}
\end{lemma}

Assertion~(3) of Theorem~\ref{T1.6} now follows; this completes the proof of Theorem~\ref{T1.6}
and thereby of all of the results of this paper.\qed

\section*{Acknowledgements}
The research of P. Gilkey was partially supported by grant PID2020-114474GB-I0 (Spain).
The research of J. H. Park was partially supported by the National Research Foundation of Korea (NRF) grant
funded by the Korea government (MSIT) (NRF-2019R1A2C1083957).
Helpful suggestions and comments were provided by our colleagues and friends
J. \'Alvarez-L\'opez, M. van den Berg, and E. Garc\'{i}a-R\'{i}o.
\section*{Dedication}
On 11 March 2004, 10 bombs exploded on 4 trains near the Atocha Station
in Madrid killing 191 and injuring more than 1800; 18 Islamic fundamentalists and
3 Spanish accomplices were convicted of the bombings which was one of Europe's deadliest
terrorist attacks in the years since World War II. Subsequently, Gilkey and his coauthors
dedicated a paper \cite{DFGG04} writing ``En memoria de todas las
v\'\i ctimas inocentes. Todos \'\i bamos en ese tre{n}.
(In memory of all these innocent victims. We were all on that train)''.
This paper is being written during one of the worst outbreaks of war in Europe since World War II.
We dedicate this paper, writing in
a similar vein to show our solidarity with the innocent victims in Ukraine, that:
  Mи всі в Україні (``we are all in Ukraine'').

\end{document}